\documentclass[11pt]{article}
\usepackage{amsmath,amsthm,amssymb}
\usepackage[hyperfootnotes=false]{hyperref}
\usepackage{color}
\usepackage{cancel}
\usepackage{titlesec}
\usepackage{enumerate}
\usepackage[normalem]{ulem}

\usepackage{geometry}
\usepackage{array}

\titleformat{\paragraph}
{\normalfont\normalsize\bfseries}{\theparagraph}{1em}{}
\titlespacing*{\paragraph}
{0pt}{3.25ex plus 1ex minus .2ex}{1.5ex plus .2ex}

\def\titlerunning#1{\gdef\titrun{#1}}
\makeatletter
\def\author#1{\gdef\autrun{\def\and{\unskip, }#1}\gdef\@author{#1}}
\def\address#1{{\def\and{\\\hspace*{18pt}}\renewcommand{\thefootnote}{}%
\footnote {#1}}% 
\markboth{\autrun}{\titrun}}
\makeatother
\def\email#1{\hspace*{4pt}{\em e-mail}: #1}
\def\MSC#1{{\renewcommand{\thefootnote}{}%
\footnote{\emph{Mathematics Subject Classification (2020):} #1}}}
\def\keywords#1{\par\medskip
\noindent\textbf{Keywords:} #1}

\newtheorem{theorem}{Theorem}[section]

\newtheorem{proposition}[theorem]{Proposition}

\newtheorem{lemma}[theorem]{Lemma}

\newtheorem{definition}[theorem]{Definition}

\theoremstyle{definition}

\newtheorem{construction}[theorem]{Construction}
\newtheorem{remark}[theorem]{Remark}

\numberwithin{equation}{section}

\def\cA{\mathcal A}

\def\cF{\mathcal F}
\def\cG{\mathcal G}
\def\cH{\mathcal H}

\def\cP{\mathcal P}
\def\cQ{\mathcal Q}

\def\cS{\mathcal S}

\def\cU{\mathcal U}

\def\cW{\mathcal W}

\def\PG{{\rm PG}}
\def\AG{{\rm AG}}

\begin{document}

%%%%% To ease editing, add:

\baselineskip=16pt

%%%%%%%%%%%%%%%%
%% In the running head, give an abbreviation of the title. 
\titlerunning{}

\title{Affine vector space partitions and spreads of quadrics}

\author{Somi Gupta
\and
Francesco Pavese}

\date{}

\maketitle

\address{S. Gupta: Department of Mathematics and Applications ``R. Caccioppoli'', University of Naples ``Federico II'', Via Cintia, Monte S. Angelo, I-80126 Naples,
Italy; \email{somi.gupta@unina.it}.
\and
F. Pavese: Department of Mechanics, Mathematics and Management, Polytechnic University of Bari, Via Orabona 4, I-70125 Bari, Italy; \email{francesco.pavese@poliba.it}.
}

%\bigskip

%%%%%%%%
\MSC{Primary: 51E23 Secondary: 05B40.}

\begin{abstract}
An {\em affine spread} is a set of subspaces of $\AG(n, q)$ of the same dimension that partitions the points of $\AG(n, q)$. Equivalently, an {\em affine spread} is a set of projective subspaces of $\PG(n, q)$ of the same dimension which partitions the points of $\PG(n, q) \setminus H_{\infty}$; here $H_{\infty}$ denotes the hyperplane at infinity of the projective closure of $\AG(n, q)$. Let $\cQ$ be a non degenerate quadric of $H_\infty$ and let $\Pi$ be a generator of $\cQ$, where $\Pi$ is a $t$-dimensional projective subspace. An affine spread $\cP$ consisting of $(t+1)$-dimensional projective subspaces of $\PG(n, q)$ is called {\em hyperbolic, parabolic} or {\em elliptic} (according as $\cQ$ is hyperbolic, parabolic or elliptic) if the following hold: 
\begin{itemize}
\item each member of $\cP$ meets $H_\infty$ in a distinct generator of $\cQ$ disjoint from $\Pi$; 
\item elements of $\cP$ have at most one point in common;
\item if $S, T \in \cP$, $|S \cap T| = 1$, then $\langle S, T \rangle \cap \cQ$ is a hyperbolic quadric of $\cQ$. 
\end{itemize}
In this note it is shown that a hyperbolic, parabolic or elliptic affine spread of $\PG(n, q)$ is equivalent to a spread of $\cQ^+(n+1, q)$, $\cQ(n+1, q)$ or $\cQ^-(n+1, q)$, respectively.

\keywords{Finite geometry; vector space partition; spread; quadric.}
\end{abstract}

\section{Introduction}
Establishing the existence of a partition of a geometric structure into subspaces of the same dimension is a classical theme in finite geometry. A {\em vector space partition} of the projective space $\PG(n, q)$ is a set of projective subspaces of $\PG(n, q)$ which partitions the points of $\PG(n, q)$. A vector space partition whose subspaces are of the same dimension is called {\em spread}. It was shown by J.~Andr\'e \cite{A} and B.~Segre \cite{S} that $\PG(n, q)$ admits a spread consisting of $t$-dimensional projective subspaces if and only if $t+1$ divides $n+1$. A finite classical polar space arises from a vector space of finite dimension over a finite field equipped with a non-degenerate reflexive sesquilinear or quadratic form. A non-degenerate polar space of $\PG(n, q)$ is a member of one of the following classes: a symplectic space $\cW(n, q)$, $n$ odd, a parabolic quadric $\cQ(n, q)$, $n$ even, a hyperbolic quadric $\cQ^+(n, q)$, $n$ odd, an elliptic quadric $\cQ^-(n, q)$, $n$ odd, or a Hermitian variety $\cH(n, q)$, $q$ a square. A {\em spread} of a polar space $\cA$ is a set of generators (i.e., maximal totally isotropic subspaces or maximal totally singular subspaces) of $\cA$, which partitions the points of $\cA$. A spread of $\cW(n, q)$ is also a spread of $\PG(n, q)$ into $\left(\frac{n-1}{2}\right)$-dimensional projective subspaces. It is easily seen that $\cQ^+(n, q)$ has no spread if $n \equiv 1 \pmod{4}$. Many authors investigated spreads of polar spaces, see \cite{CKW, Dye, K1, K2, K3, M, Shult, Thas}. For $q$ even, $\cQ^+(n, q)$, $n \equiv -1\pmod{4}$, $\cQ(n, q)$, $\cQ^-(n, q)$ always have a spread. For $q$ odd, with $q$ a prime or $q \equiv 0$ or $2 \pmod{3}$, $\cQ^+(7, q)$ and $\cQ(6, q)$ have a spread. The parabolic quadric $\cQ(n, q)$, with $n \equiv 0 \pmod{4}$ and $q$ odd, has no spread. For $q$ odd, $\cQ^+(3, q)$ and $\cQ^-(5, q)$ have a spread. The Hermitian varieties $\cH(n, q)$, $n$ odd and $\cH(4, 4)$ do not have a spread. For open problems related to spreads of polar spaces, the reader is referred to \cite[Section 7.5]{H3}

In this context it is natural to consider an ``affine version'' of a vector space partition as a set of proper affine subspaces that partitions the points of $\AG(n, q)$, see \cite{Bam}. Denote by $H_{\infty}$ the hyperplane at infinity of the projective closure of $\AG(n, q)$. Then an {\em affine vector space partition} of $\PG(n, q)$ is a set of projective subspaces of $\PG(n, q)$ which partitions the points of $\PG(n, q) \setminus H_{\infty}$. An {\em affine spread} is an affine vector space partition whose subspaces are of the same dimension. Let $\cQ$ be a non degenerate quadric of $H_\infty$ and let $\Pi$ be a generator of $\cQ$, where $\Pi$ is a $t$-dimensional projective subspace. Here, we are concerned with affine spreads $\cP$ consisting of $(t+1)$-dimensional projective subspaces of $\PG(n, q)$ such that 
\begin{itemize}
\item each member of $\cP$ meets $H_\infty$ in a distinct generator of $\cQ$ disjoint from $\Pi$; 
\item elements of $\cP$ have at most one point in common;
\item if $S, T \in \cP$, $|S \cap T| = 1$, then $\langle S, T \rangle \cap \cQ$ is a hyperbolic quadric of $\cQ$. 
\end{itemize}
An affine spread $\cP$ of $\PG(n, q)$ satisfying the above properties is called {\em hyperbolic, parabolic} or {\em elliptic}, according as $\cQ$ is hyperbolic, parabolic or elliptic, respectively. In \cite[Section 5.2]{Bam} the authors exibhited a particular  hyperbolic affine spread of $\PG(6, q)$, $q$ even or $q \in \{3, 5\}$, and they conjecture that a hyperbolic affine spread of $\PG(6, q)$ exists for all prime powers \cite[Conjecture 2]{Bam}. In this note it is shown that a hyperbolic, parabolic or elliptic affine spread of $\PG(n, q)$ is equivalent to a spread of $\cQ^+(n+1, q)$, $\cQ(n+1, q)$ or $\cQ^-(n+1, q)$, respectively.

\section{Affine vector space partitions and quadrics}

Throughout the paper we will use the term {\em $s$-space} to denote an $s$-dimensional projective subspace of the ambient projective space. Let $\cQ_{r, e}$ denote a non-degenerate quadric of $\PG(r, q)$ as indicated below:
\begin{align*}
\begin{tabular}{c||c|c|c}
$\cQ_{r, e}$ & $\cQ^+(r, q)$ & $\cQ(r, q)$ & $\cQ^-(r, q)$ \\
\hline
  $e$ & 0 & $1$ & $2$ \\
%\hline
%  $d$ & $(r-1)/2$ & $(r-2)/2$ & $(r-3)/2$ 
\end{tabular}, 
\end{align*}
where $r$ is odd if the quadric is hyperbolic or elliptic, whereas $r$ is even if the quadric is parabolic. Associated with $\cQ_{r, e}$, there is a polarity $\perp$ of $\PG(r, q)$, which is non-degenerate except when $e = 1$ and $q$ is even. In particular, the polarity $\perp$ is symplectic if $\cQ_{r, e} \in \{\cQ^+(r, q), \cQ^-(r, q)\}$, $q$ even, and orthogonal if $\cQ_{r, e} \in \{\cQ(r, q), \cQ^+(r, q), \cQ^-(r, q)\}$, $q$ odd. For $\cQ_{r, 1} = \cQ(r, q)$, $q$ even, the polarity $\perp$ is degenerate, indeed $N^{\perp} = \PG(r, q)$, if $N$ is the nucleus of $\cQ(r, q)$, whereas $P^\perp$ is a hyperplane of $\PG(r, q)$ for any other point $P$ of $\PG(r, q)$. A {\em generator} of $\cQ_{r, e}$ is a projective space of maximal dimension contained in $\cQ_{r, e}$ and generators of $\cQ_{r, e}$ are $\left( \frac{r-e-1}{2} \right)$-spaces. A {\em spread} of $\cQ_{r, e}$ is a set of $q^{\frac{r+e-1}{2}}+1$ pairwise disjoint generators of $\cQ_{r, e}$. More background information on the properties of the finite classical polar spaces can be found in \cite{H2, H1, H3}.

\begin{definition}
Let $H$ be a hyperplane of $\PG(r+1, q)$. An {\em affine vector space partition} (or abbreviated as {\em avsp}) of $\PG(r+1,q)$ is a set $\cP$ of subspaces of $\PG(r+1,q)$ whose members are not contained in $H$ and such that every point of $\PG(r+1, q) \setminus H$ is contained in exactly one element of $\cP$.
\end{definition}

The {\em type} of an avsp $\cP$ of $\PG(r+1, q)$ is given by $(r+1)^{m_{r+1}} \dots 2^{m_2} 1^{m_1}$, where $m_i$ denotes the number of $(i-1)$-spaces of $\cP$ for $1\leq i \leq r+1$.

\begin{definition}
    Let $\cP$ be an avsp of $\PG(r+1,q)$. Then $\cP$ is said to be {\em reducible} if there exists a proper subspace $U$ of $\PG(r+1, q)$ such that the members of $\cP$ contained in $U$ form an avsp of $U$. If $\cP$ is not reducible, then it is said to be {\em irreducible}.
\end{definition}

\begin{definition}
    Let $\cP$ be an avsp of $\PG(r+1,q)$. Then $\cP$ is said to be {\em tight} if no point of $\PG(r+1, q)$ belongs to each of the members of $\cP$.
\end{definition}

In \cite[Section 5.2]{Bam} the authors studied a particular avsp $\{S_1, \dots, S_{q^3}\}$ of $\PG(6, q)$, $q$ even, of type $4^{q^3}$, such that the set $\{S_i \cap H \mid i = 1, \dots, q^3\}$ consists of the $q^3$ planes of a Klein quadric $\cQ^+(5, q)$ that are disjoint from a fixed plane of $\cQ^+(5, q)$. Moreover, they conjecture that a similar construction can be realized for all prime powers \cite[Conjecture 2]{Bam}. Motivated by their example, we introduce the definition of {\em hyperbolic, parabolic, or elliptic avsp}.

\begin{definition}\label{def}
Let $\cP = \left\{S_1, S_2, \dots, S_{q^{\frac{r+e+1}{2}}}\right\}$ be an avsp of $\PG(r+1, q)$ of type $\left(\frac{r-e+3}{2}\right)^{q^{\frac{r+e+1}{2}}}$. Then $\cP$ is said {\em hyperbolic, parabolic or elliptic} if the set 
\begin{align*}
\cG = \left\{\Pi_i = S_i \cap H \mid i = 1, \dots, q^{\frac{r+e+1}{2}}\right\}
\end{align*}
consists of $q^{\frac{r+e+1}{2}}$ generators of a quadric $\cQ_{r, e} \subset H$, where $e = 0, 1, 2$, respectively, such that 
\begin{enumerate}
\item elements of $\cG$ are disjoint from a fixed generator $\Pi$ of $\cQ_{r, e}$;
\item distinct members of $\cG$ (or of $\cP$) meet at most in one point;
\item if $|\Pi_i \cap \Pi_j| = 1$, then $\langle S_i, S_j \rangle \cap \cQ_{r, e} = \cQ^+(r - e, q)$.
\end{enumerate}
\end{definition}

\begin{remark}\label{remark1}
Let $r$ be odd, let $\cP = \left\{S_1, S_2, \dots, S_{q^{\frac{r+1}{2}}}\right\}$ be a hyperbolic avsp of $\PG(r+1, q)$ and let $\Pi_i = S_i \cap H$, $i = 1, \dots, q^{\frac{r+1}{2}}$. Since members of $\cP$ are $\left( \frac{r+1}{2} \right)$-spaces, then $\langle S_i, S_j \rangle = \PG(r+1, q)$ and $|S_i \cap S_j| = 1$, if $i \ne j$. Hence $|\Pi_i \cap \Pi_j| = 1$, if $i \ne j$. Therefore for the definition of a hyperbolic avsp {\em 2)} is equivalent to require that distinct members of $\cG$ pairwise intersect in one point, whereas requirement {\em 3)} is redundant.  

If $r \equiv -1 \pmod{4}$, then members of $\cG$ and $\Pi$ belong to the same system of generators and hence $|\Pi_i \cap \Pi_j| = 0$, if $i \ne j$. Therefore hyperbolic avsp of $\PG(r+1, q)$, $r \equiv -1\pmod{4}$, do not exist. 
\end{remark}

\begin{remark}
If $r = e+1$, then there are not $q^{e+1}$ generators of $\cQ_{r, e}$ disjoint from $\Pi$. Hence there is no hyperbolic, parabolic or elliptic avsp for $(r, e) \in \{(1, 0), (2, 1), (3, 2)\}$. If $r = e+3$, then requirement {\em 2)} is redundant.
\end{remark}

\begin{remark}
The existence of a hyperbolic avsp of $\PG(6, q)$ is equivalent to the ``extension problem'' formulated in \cite[Conjecture 2]{Bam}. 
\end{remark}

In the following we will need the next result, that in the hyperbolic case has been proved in \cite[Example 7.6]{Kantor}.

\begin{lemma}\label{aux}
Let $\left\{\Sigma_1, \dots, \Sigma_{q^{\frac{r+e+1}{2}}+1}\right\}$ be a spread of a quadric $\cQ_{r+2, e}$ of $\PG(r+2, q)$. Fix a point $P \in \Sigma_{q^{\frac{r+e+1}{2}}+1}$ and an $r$-space $H \subset P^\perp$ such that $P \notin H$. Set $\cQ_{r, e} = H \cap \cQ_{r+2, e}$ and 
    \begin{align*}
        \Pi_i = \langle P, \Sigma_i \rangle \cap H, \quad i = 1, \dots, q^{\frac{r+e+1}{2}}.
    \end{align*}
    The following hold.
    \begin{itemize}
        \item[i)] If $e \in \{1, 2\}$, i.e., $\cQ_{r+2, e}$ is parabolic or elliptic, then $\Pi_i$ and $\Pi_j$, $i \ne j$, intersect in at most one point.
        \item[ii)] If $e = 0$, i.e., $\cQ_{r+2, 0}$ is hyperbolic and $r \equiv 1 \pmod{4}$, then $\Pi_i$ and $\Pi_j$, $i \ne j$, have exactly one point in common. 
        \item[iii)] Each point of $\cQ_{r, e} \setminus \Sigma_{q^{\frac{r+e+1}{2}}+1}$ lies in precisely $q$ members of $\left\{\Pi_i \mid i = 1, \dots, q^{\frac{r+e+1}{2}}\right\}$. 
    \end{itemize}
\end{lemma}
\begin{proof}
Since $|(\Sigma_i\cap P^\perp)\cap(\Sigma_j\cap P^\perp)| = 0$, it follows that $\Pi_i\cap\Pi_j$ is at most one point. Assume that $e = 0$, i.e., $\cQ_{r+2, 0}$ is hyperbolic and $r \equiv 1 \pmod{4}$. Consider the $\left(\frac{r+1}{2}\right)$-space $T_i=\langle P, \Sigma_i\cap P^\perp\rangle$. Then $T_i$ is a generator of $\cQ_{r+2, 0}$ intersecting $\Sigma_i$ in an $\left(\frac{r-1}{2}\right)$-space. Therefore $T_i$ and $\Sigma_j$, $i \ne j$, are in different systems of generators and hence they have at least one point in common. It follows that $\Pi_i$ and $\Pi_j$, $i \ne j$, have at least one point in common.

Finally, it is easily observed that any line passing through $P$ meets precisely $q$ elements of the spread distinct from $\Sigma_{q^{\frac{r+e+1}{2}}+1}$. Hence, through a point of $\cQ_{r, e} \setminus \Sigma_{q^{\frac{r+e+1}{2}}+1}$ there pass precisely $q$ members of $\cG$. 
\end{proof}

In what follows we observe that a hyperbolic, elliptic, or parabolic avsp of $\PG(r+1, q)$ can be obtained starting from a spread of a hyperbolic, elliptic, or parabolic quadric of $\PG(r+2, q)$. 

\begin{construction}\label{con}
Let $\left\{\Sigma_1, \dots, \Sigma_{q^{\frac{r+e+1}{2}}+1}\right\}$ be a spread of a quadric $\cQ_{r+2, e}$ of $\PG(r+2, q)$. Fix a point $P \in \Sigma_{q^{\frac{r+e+1}{2}}+1}$ and a hyperplane $\mathcal{U} \cong \PG(r+1,q)$ not containing the point $P$. Then $H = \mathcal{U} \cap P^{\perp}$ is an $r$-space of $\PG(r+2, q)$ such that $\cQ_{r, e} = H \cap \cQ_{r+2}$. Let 
   \begin{align*}
        S_i = \langle P, \Sigma_i\rangle\cap\mathcal{U}, \quad i = 1, \dots, q^{\frac{r+e+1}{2}},
    \end{align*} 
and let $\cP = \left\{S_1, \dots, S_{q^{\frac{r+e+1}{2}}}\right\}$. Then $\cP$ consists of $q^{\frac{r+e+1}{2}}$ $\left(\frac{r-e+1}{2}\right)$-spaces of $\cU$, none of them is contained in $H$.
\end{construction}

\begin{proposition}\label{avsp0}
The set $\cP$ is a hyperbolic, parabolic or elliptic avsp of $\PG(r+1, q)$ according as $\cQ_{r+2, e}$ is hyperbolic, parabolic or elliptic, respectively.     
\end{proposition}
\begin{proof}
In order to prove that $\cP$ is an avsp of $\PG(r+1, q)$ it is enough to show that for every point of $\cU \setminus H$, there exists a member of $\cP$ containing it. Let $R$ be a point of $\mathcal{U}\setminus H$ and let $\ell$ be the line joining $P$ and $R$. Since $\ell$ passes through $P$ and is not contained in $P^{\perp}$, it is secant to $\cQ_{r+2, e}$. Therefore there is a point, say $R'$, distinct from $P$, such that $R' \in \cQ_{r+2, e}$. Moreover $R' \in \Sigma_j$, for some $j \in \left\{1, \dots, q^{\frac{r+e+1}{2}}\right\}$ and hence $R \in S_j$, by construction. 

Set $\cG = \left\{\Pi_i = S_i \cap H \mid i = 1, \dots, q^{\frac{r+e+1}{2}}\right\}$. Then $\cG$ consists of generators of $\cQ_{r, e}$ disjoint from $\Sigma_{q^{\frac{r+e+1}{2}}+1} \cap H \subset \cQ_{r, e}$. By Lemma~\ref{aux}, $|\Pi_i \cap \Pi_j|$, $i \ne j$, equals one, if $\cQ_{r+2, e}$ is hyperbolic, or is at most one, otherwise. 

Let $S_i, S_j \in \cP$, with $|S_i \cap S_j| = 1$. Then $\langle S_i, S_j, P \rangle = \langle \Sigma_i, \Sigma_j \rangle$ is an $(r+2-e)$-space of $\PG(r+2, q)$ containing two disjoint generators of $\cQ_{r+2, e}$. Hence $\langle S_i, S_j, P \rangle \cap \cQ_{r+2, e} = \cQ_{r+2-e, 0} \simeq \cQ^+(r+2-e, q)$. Such a quadric $\cQ_{r+2-e, 0}$ meets $P^\perp$ in a cone having as vertex the point $P$ and as base a $\cQ_{r-e, 0} \simeq \cQ^+(r-e, q)$. It follows that $\langle S_i, S_j \rangle \cap \cQ_{r, e} = \cQ^+(r-e, q)$, as required.  
\end{proof}

We have seen that a spread of a hyperbolic, elliptic, or parabolic quadric of $\PG(r+2, q)$ gives rise to a hyperbolic, elliptic, or parabolic avsp of $\PG(r+1, q)$, respectively. The converse also holds true, as shown below.

\begin{theorem}\label{avsp}
If $\cP$ is a hyperbolic, parabolic or elliptic avsp of $\PG(r+1, q)$, then there is a spread of $\cQ_{r+2, e}$, where $\cQ_{r+2, e}$ is a hyperbolic, parabolic or elliptic quadric, respectively.    
\end{theorem}
\begin{proof}
Let $\cP = \left\{S_1, \dots, S_{q^{\frac{r+e+1}{2}}}\right\}$ be a hyperbolic, parabolic or elliptic avsp of $\PG(r+1, q)$. Then there exists a hyperplane $H$ of $\PG(r+1, q)$, a quadric $\cQ_{r, e}$ of $H$, that is hyperbolic, parabolic or elliptic, respectively, and a fixed generator $\Pi$ of $\cQ_{r, e}$, such that $\Pi_i = S_i \cap H$ is a generator of $\cQ_{r, e}$ disjoint from $\Pi$. Embed $\PG(r+1, q)$ as a hyperplane section, say $\cU$, of $\PG(r+2, q)$ and fix a quadric $\cQ_{r+2, e}$ of $\PG(r+2, q)$ in such a way that $H \cap \cQ_{r+2, e} = \cQ_{r, e}$. Let $P$ be one of the two points of $\cQ_{r+2, e}$ on $H^\perp$, where $\perp$ is the polarity of $\PG(r+2, q)$ associated with $\cQ_{r+2, e}$. For $i = 1, \dots, q^{\frac{r+e+1}{2}}$, consider the following $\left(\frac{r-e+3}{2}\right)$-space
\begin{align*}
\cF_i = \langle P, S_i \rangle.
\end{align*}
Hence $\cF_i$ meets $P^\perp$ in the $\left(\frac{r-e+1}{2}\right)$-space spanned by $P$ and $\Pi_i$, that is a generator of $\cQ_{r+2, e}$. We claim that $\cF_i$ contains a further generator of $\cQ_{r+2, e}$ besides $\langle P, \Pi_i \rangle$. Indeed, if $\langle P, \Pi_i \rangle$ were the unique generator of $\cQ_{r+2, e}$ contained in $\cF_i$, then $\cF_i \subset \langle P, \Pi_i \rangle^\perp = P^\perp \cap \Pi_i^\perp \subset P^\perp$, contradicting the fact that $\cF_i \cap P^\perp$ is an $\left(\frac{r-e+1}{2}\right)$-space. Therefore $\cF_i$ has to contain a further generator of $\cQ_{r+2, e}$, say $\Sigma_i$. Denote by $\Sigma_{q^{\frac{r+e+1}{2}}+1}$ the generator of $\cQ_{r+2, e}$ spanned by $P$ and $\Pi$. We claim that 
\begin{align*}
\left\{\Sigma_1, \dots, \Sigma_{q^{\frac{r+e+1}{2}}+1}\right\}
\end{align*}
is a spread of $\cQ_{r+2, q}$. Assume by contradiction that there is a point $Q' \in \Sigma_i \cap \Sigma_{q^{\frac{r+e+1}{2}}+1}$, for some $i \in \left\{1, \dots, q^{\frac{r+e+1}{2}}\right\}$. Let $Q = \langle P, Q' \rangle \cap H$, then $Q \in \Pi_i \cap \Pi$, a contradiction. Assume by contradiction that $|\Sigma_i \cap \Sigma_j| > 0$, for some $1 \le i < j \le q^{\frac{r+e+1}{2}}$. Then necessarily $\Sigma_i \cap \Sigma_j$ is a point, otherwise $|S_i \cap S_j| > 1$. Moreover, we infer that such a point, say $R' = \Sigma_i \cap \Sigma_j$, must be in $P^\perp$. Indeed, if $R' \notin P^\perp$, then the point $R = \langle P, R' \rangle \cap \cU$ belongs to both $S_i \setminus H$ and $S_j \setminus H$, contradicting the fact that $S_1, \dots, S_{q^{\frac{r+e+1}{2}}}$ is an avsp. In particular, $|S_i \cap S_j| = 1$.

If $e = 0$, by Remark~\ref{remark1}, we may assume that $r \equiv 1 \pmod{4}$. Hence $r+2 \equiv -1 \pmod{4}$ and $\Sigma_k$ and $\Sigma_{q^{\frac{r+e+1}{2}}+1}$ belong to the same system of generators of $\cQ_{r+2, 0}$, for $k = 1, \dots, q^{\frac{r+e+1}{2}}$. It follows that if $|\Sigma_i \cap \Sigma_j| > 0$, then they have at least a line in common, a contradiction.

If $e \in \{1, 2\}$, observe that $\langle \cF_i, \cF_j \rangle$ is a $\PG(r+2-e, q)$ containing the cone having as vertex the point $P$ and as base the hyperbolic quadric $\cQ^+(r-e, q) = \langle S_i, S_j \rangle \cap \cQ_{r, e}$. Furthermore, two more generators of $\cQ_{r+2, e}$, namely $\Sigma_i, \Sigma_j$, are contained in $\langle \cF_i, \cF_j \rangle$ and do not pass through $P$. Therefore necessarily we have that $\langle \cF_i, \cF_j \rangle \cap \cQ_{r+2, e} = \cQ_{r+2-e, 0}$. Let us denote by $\cQ^+(r+2-e, q)$ the hyperbolic quadric $\langle \cF_i, \cF_j \rangle \cap \cQ_{r+2, e}$, so that generators of $\cQ^+(r+2-e, q)$ are $\left(\frac{r-e+1}{2}\right)$-spaces. Since $\Sigma_i \cap \langle P, \Pi_i \rangle$ and $\Sigma_j \cap \langle P, \Pi_j \rangle$ are $\left(\frac{r-e-1}{2}\right)$-spaces, we have that $\Sigma_i$ and $\langle P, \Pi_i \rangle$ lie in distinct systems of generators of $\cQ^+(r+2-e)$. Similarly for $\Sigma_j$ and $\langle P, \Pi_j \rangle$. Two possibilities arise: either $r+2-e \equiv -1 \pmod{4}$ or $r+2-e \equiv 1 \pmod{4}$. Since $\langle P, \Pi_i \rangle \cap \langle P, \Pi_j \rangle$ is a line, if the former case occurs, then $\langle P, \Pi_i \rangle$, $\langle P, \Pi_j \rangle$ belong to the same system of generators of $\cQ^+(r+2-e, q)$. Hence $\Sigma_i$, $\Sigma_j$ belong to the same system of generators of $\cQ^+(r+2-e, q)$ and if $|\Sigma_i \cap \Sigma_j| > 0$, then they have at least a line in common, a contradiction. In the latter case, $\langle P, \Pi_i \rangle$, $\langle P, \Pi_j \rangle$ are in different systems of generators of $\cQ^+(r+2-e, q)$. Therefore $\Sigma_i$, $\langle P, \Pi_j \rangle$ belong to the same system of generators of $\cQ^+(r+2-e, q)$. Similarly for $\Sigma_j$, $\langle P, \Pi_i \rangle$. It follows that $\Sigma_i$, $\Sigma_j$ belong to different systems of generators of $\cQ^+(r+2-e, q)$ and again, if $|\Sigma_i \cap \Sigma_j| > 0$, then they have at least a line in common, a contradiction.
\end{proof}

As pointed out in \cite[Theorem 42]{Bam}, tightness and irreducibility follow immediately for a hyperbolic avsp of $\PG(5, q)$. We will show that the same occurs in higher dimensions and for parabolic or elliptic avsp.

\begin{lemma}\label{aux1}
Let $e \in \{1, 2\}$, let $\cS = \left\{ \Sigma_1, \dots, \Sigma_{q^{\frac{r+e+1}{2}}+1} \right\}$ be a spread of $\cQ_{r+2, e}$. Then at most $q+1$ members of $\cS$ are contained in a quadric $\cQ_{r+1,e-1} \subset \cQ_{r+2, e}$.
\end{lemma}
\begin{proof}
If $e = 1$, then exactly $2$ members of $\cS$ are contained in a hyperbolic hyperplane section $\cQ_{r+1, 0} \simeq \cQ^+(r+1, q)$ of $\cQ_{r+2, 1} \simeq \cQ(r+2, q)$, whereas, if $e = 2$, then there are $q+1$ elements of $\cS$ contained in a parabolic hyperplane section $\cQ_{r+1, 1} \simeq \cQ(r+1, q)$ of $\cQ_{r+2, 2} \simeq \cQ^-(r+2, q)$. %Therefore if $\cQ^+(r, q)$ is a hyperbolic quadric contained in $\cQ^-(r+2, q)$, then no more of $q+1$ members of $\cS$ are contained in $\cQ^+(r, q)$. 
\end{proof}

\begin{proposition}\label{tight-irreducibility}
Let $\cP$ be a hyperbolic, parabolic or elliptic avsp of $\PG(r+1, q)$, then $\cP$ is tight and irreducible.
\end{proposition}
\begin{proof}
Let $\cP=\left\{S_1, S_2, \dots, S_{q^{\frac{r+e+1}{2}}}\right\}$ be a hyperbolic, parabolic or elliptic avsp of $\PG(r+1, q)$. If $e = 0$, in order to prove that $\cP$ is irreducible, it is enough to observe that, if $i \ne j$, then the span of $S_i$ and $S_j$ is the whole $\PG(r+1, q)$. Similarly, if $e = 1$ and $|S_i \cap S_j| = 0$, then $\langle S_i, S_j \rangle = \PG(r+1, q)$. If either $e = 1$ and $|S_i \cap S_j| = 1$ or $e = 2$, we claim that the elements of $\cP$ contained in $\langle S_i, S_j \rangle$, where $i \ne j$, do not cover all the points of $\langle S_i, S_j \rangle \setminus H$. By Theorem~\ref{avsp}, there is a quadric $\cQ_{r+2, e}$ with a spread $\cS = \left\{ \Sigma_1, \dots, \Sigma_{q^{\frac{r+e+1}{2}}+1} \right\}$ such that $\cP$ can be obtained via Construction~\ref{con}. Then the number of elements of $\cP$ contained in $\langle S_i, S_j \rangle$ equals the number of elements of $\cS$ contained in $\langle P, \Sigma_i, \Sigma_j \rangle$. Since $\langle P, \Sigma_i, \Sigma_j \rangle \cap \cQ_{r+2, e}$ is either a hyperbolic quadric $\cQ^+(r+2-e, q)$, if $|S_i \cap S_j| = 1$, or a parabolic quadric $\cQ(r+1, q)$, if $|S_i \cap S_j| = 0$ and $e = 2$, such a number cannot exceed $q+1$ by Lemma~\ref{aux1}. It follows that $\cP$ is irreducible.

By Lemma~\ref{aux}~{\em iii)}, through a point of $H$ there pass precisely $q$ members $\cP$. Hence tightness follows. 
\end{proof}

\section{Conclusion}

We have seen that a hyperbolic, parabolic or elliptic avsp of $\PG(n, q)$ is equivalent to a spread of $\cQ^+(n+1, q)$, $\cQ(n+1, q)$ or $\cQ^-(n+1, q)$, respectively. Furthermore, such an avsp is tight and irreducible. Based on Proposition~\ref{avsp0} and Theorem~\ref{avsp}, the extension problem formulated in \cite[Conjecture 2]{Bam} is equivalent to that of the existence of a spread of the triality quadric $\cQ^+(7, q)$, which in turn is equivalent to that of the existence of an ovoid of $\cQ^+(7, q)$. We remark that the existence of a spread of $\cQ^+(7, q)$ has been established in the cases when $q$ is even or when $q$ is odd, with $q$ a prime or $q \equiv 0$ or $2 \pmod {3}$.

\bigskip

\noindent\textit{Conflict of interest statement.}
The author declare that there is no conflict of interest.

\bigskip

\noindent\textit{Data availability statement.}
Data sharing not applicable to this article as no datasets were generated or analysed during the current study.

\bigskip

\smallskip
{\footnotesize
\noindent\textit{Acknowledgments.}
This work was supported by the Italian National Group for Algebraic and Geometric Structures and their Applications (GNSAGA-- INdAM).
}

\end{document}